\documentclass[a4paper,11pt]{amsart}

\usepackage[top=40mm, bottom=40mm, left=30mm, right=30mm]{geometry} 
\usepackage{nicefrac} 

\usepackage[dvipsnames]{xcolor} 

\usepackage{mathrsfs} 

\usepackage{amsmath} 

\usepackage{amsfonts} 

\usepackage{euscript} 

\usepackage{amssymb} 
\usepackage{euscript} 

\usepackage{enumerate} 
\usepackage{mathtools} 
\usepackage[autostyle]{csquotes} 
\usepackage[nameinlink, capitalise, noabbrev]{cleveref} 
\usepackage{bbm}

\newcommand{\sign}{\operatorname{sign}} 
\newcommand{\rg}{\operatorname{rg}} 
\renewcommand{\Re}{\operatorname{Re}}

\newcommand{\dom}{\operatorname{dom}} 
\newcommand{\hsign}{\widehat{\sign}} 
\newcommand{\euS}{\EuScript{S}} 
\newcommand{\euT}{\EuScript{T}}

\newcommand{\defeq}{\mathrel{\vcentcolon=}} 
 
\newcommand{\mmid}{\,\middle|\,}

\newcommand{\ee}{\mathrm{e}} 
\newcommand{\dd}{\,\mathrm{d}} 
\newcommand{\Ell}{\mathrm{L}} 
\newcommand{\Ce}{\mathrm{C}} 
\newcommand{\Me}{\mathrm{M}} 

\newcommand{\Cb}{\Ce_\mathrm{b}} 
\newcommand{\Id}{\mathrm{Id}} 
 
\newcommand{\C}{\mathbb{C}} 
\newcommand{\N}{\mathbb{N}} 
\newcommand{\R}{\mathbb{R}}

\newcommand{\LLL}{\mathscr{L}} 
\newcommand{\calB}{\mathscr{B}} 
\newcommand{\pow}{\mathscr{P}}

\theoremstyle{plain}

\newtheorem{theorem}{Theorem} 
\newtheorem{lemma}[theorem]{Lemma} 
 
\newtheorem{proposition}[theorem]{Proposition}

\theoremstyle{definition} 
\newtheorem{remark}[theorem]{Remark} 
\newtheorem{example}[theorem]{Example} 
\newtheorem{definition}[theorem]{Definition}
\begin{document}

\title[Koopman semigroups on $\Cb(X)$]{Towards a Koopman theory for dynamical systems on completely regular spaces}


\author{Bálint Farkas}

\thanks{The authors thank Wolfgang Arendt, Markus Haase, Viktoria Kühner, Rainer Nagel, and Abdelaziz Rhandi for ideas and advice.} 
\address{Fakult\"at f\"ur Mathematik und Naturwissenschaften, Bergische Universität Wuppertal, Gauß\-straße 20, D-42119 Wuppertal, Germany}

\email{farkas@math.uni-wuppertal.de} 
\author{Henrik Kreidler}

\address{Fakult\"at f\"ur Mathematik und Naturwissenschaften, Bergische Universität Wuppertal, Gauß\-straße 20, D-42119 Wuppertal, Germany} 
\email{kreidler@uni-wuppertal.de}

\def\balint#1{{\color{blue}#1}} 
\def\henrik#1{{\color{orange}#1}}
\begin{abstract}
	The Koopman linearization of measure-preserving systems or topological dynamical systems on compact spaces has proven to be extremely useful. In this article we look at dynamics given by continuous semiflows on completely regular spaces which arise naturally from solutions of PDEs. We introduce Koopman semigroups for these semiflows on spaces of bounded continuous functions. As a first step we study their continuity properties as well as their infinitesimal generators. We then characterize them algebraically (via derivations) and lattice theoretically (via Kato's equality). Finally, we demonstrate---using the example of attractors---how this Koopman approach can be used to examine properties of dynamical systems. 
\end{abstract}

\maketitle The idea to investigate properties of a dynamical system via the time-evolution of its observables goes all to the way back to the beginnings of thermodynamics and statistical mechanics at the end of the 19th century. Boltzmann and others observed that it is impossible to exactly determine the position and momentum of each and every particle of a gas or fluid. So it is a physical necessity to observe how measurement parameters such as temperature or pressure evolve over time. This basic, but important idea can be translated into the mathematical analysis of dynamical systems: Instead of examining a semigroup action $\varphi$ on a state space $X$, we pass to a suitable vector space $\mathrm{F}(X)$ of functions on $X$ and the induced semigroup $\euT_\varphi$ thereon. The advantage of this procedure is the gain of algebraic structure. Most importantly, we arrive at a semigroup of \emph{linear} mappings on $\mathrm{F}(X)$, and hence the transition from the original system to this semigroup can be seen as a global linearization.

This is an approach used in many mathematical disciplines, e.g., in differential geometry (differentiable dynamics on a smooth manifold induce operators on the space of smooth functions), complex analysis (holomorphic mappings induce operators on Banach and Fréchet spaces of holomorphic functions), algebraic geometry (morphisms of algebraic varieties induce mappings on algebras of regular functions) and ergodic theory (measure-preserving dynamics induce operators on the corresponding $\Ell^p$-spaces). Depending on the respective mathematical background, the arising linear maps are known as \emph{composition operators}, \emph{pullbacks}, \emph{induced operators} or \emph{Koopman operators} (named after B.~O.~Koopman who introduced them in the context of ergodic theory in \cite{Koop1931}). 

In this article we are interested in the Koopman linearization of topological dynamical systems, i.e., continuous semiflows on topological spaces. In the case of dynamics on a compact space, this approach is systematically pursued in \cite{EFHN2015} for the time-discrete case and---to some extent---in \cite{PosOp1986} for the time-continuous one. The heart of this \enquote{Koopman theory} for topological dynamical systems are the elegant results on Banach algebras and Banach lattices allowing a translation of properties of topological dynamical systems to the language of functional analysis. In particular, the famous representation theorems of Gelfand for commutative C$^*$-algebras (see, e.g., \cite[Section 1.4]{Dixm1977}) and Kakutani for AM-spaces (see, e.g., \cite[Section II.7]{Scha1974}) imply that the categories of topological dynamical systems on compact spaces on one hand, and strongly continuous semigroups of unital algebra or lattice homomorphisms on unital commutative C$^*$-algebras or on unital AM-spaces on the other hand, are equivalent.

We are interested in extending the Koopman theory to the more general framework of semiflows on completely regular spaces. First results in this direction have been given in \cite{DoNe1993}, \cite{DoNe1996} \cite{DoNe2000}, \cite{Kueh2001} and \cite{Kueh2020} for dynamics on metric spaces. The interest in this very general class of topological dynamical systems stems from the fact that they often arise from the solutions of abstract Cauchy problems on a Banach space $E$ of the form 
\begin{align*}
	u'(t) &= Au(t) \textrm{ for } t \geq 0,\\
	u(0) &= u_0 
\end{align*}
where $A \colon \dom(A) \rightarrow E$ is a (generally non-linear) map defined on a subset $\dom(A)$ of $E$ and $u_0 \in \dom(A)$ is a fixed initial value. These, in term, can be used to give a functional analytic description of (non-linear) evolution equations (see, e.g., \cite{Miya1992}, \cite{Tema1998}, \cite{Robi2001} and \cite{GLMY2018}). 

In this article we describe a Koopman linearization of dynamical systems on completely regular spaces (e.g., subsets of Banach spaces). After giving an introduction as well as a discussion of examples and basic continuity properties in Section 1, we proceed by extending the algebraic and lattice theoretic characterizations of Koopman semigroups from the compact case to our more general situation in Section 2. In the final section we give an outlook how Koopmanism can be applied to obtain a new approach to key dynamical concepts such as attractors.

In the following all (locally) compact spaces, completely regular spaces and topological vector spaces are assumed to be Hausdorff.

\section{Koopman linearization of semiflows on completely regular spaces} Topological dynamical systems have been studied extensively if the underlying space is compact (see, e.g. \cite{GoHe1955}, \cite{Ausl1988}, \cite{Bron1979} and \cite{deVr1993}). Here we start from a more general setting and consider dynamics on a non-empty completely regular space $X$.
\begin{definition}
	A family $\varphi = (\varphi_t)_{t \geq 0}$ of continuous mappings on $X$ is a \emph{semiflow} if 
	\begin{itemize}
		\item $\varphi_s \circ \varphi_t = \varphi_{s+t}$ for all $s,t \geq 0$, and 
		\item $\varphi_0 = \mathrm{id}_{X}$. 
	\end{itemize}
	It is \emph{(jointly) continuous} if the map 
	\begin{align*}
		\R_{\geq 0} \times X \rightarrow X, \quad (t,x) \mapsto \varphi_t(x) 
	\end{align*}
	is continuous. In this case, we call the pair $(X;\varphi)$ a \emph{topological dynamical system}\footnote{Some authors use the term \emph{semi-dynamical system} for obvious reasons.}. 
\end{definition}
A simple example of a continuous semiflow on $X = [0,\infty)$ is given by $\varphi_t(x) \defeq x + t$ for $(t,x) \in \R_{\geq 0}^2$. More interesting ones arise from partial differential equations which are formulated as abstract Cauchy-problems on Banach or Hilbert spaces. 
\begin{example}
	[$m$-accretive relations, the Crandall--Liggett theorem] Let $E$ be a Banach space and $A\subseteq E\times E$ be a relation such that $\Id+\omega A$ is accretive for some $\omega\in \R$, meaning that $(\Id+\omega A)^{-1}$ is a $1$-Lipschitz continuous function on $\dom ((\Id+\omega A)^{-1})$. Set $X:=\overline{\dom(A)}\subseteq E$. The Crandall--Liggett theorem states that if $X\subseteq \rg(\Id+\lambda A)$ for all sufficiently small $\lambda>0$, then for every $x\in X$ and $t\geq 0$ the limit
	\[ \varphi(t,x):=\lim_{k\to\infty} \left(\Id+\frac{t}{k}\right)^{-k}x \]
	exists and defines a dynamical system $\varphi:\R_{\geq 0}\times X\to X$ on $X$. See \cite{CranLig} and, e.g., \cite{BCP}, \cite{Barbu}, \cite{Miya1992} for details and relations to differential equations. 
\end{example}
\begin{example}[Analytic semigroups and semilinear equations]\label{examp:analytic} 
	Let $E$ be a Banach space and let $A$ be an invertible linear operator such that $-A$ is the generator of a bounded, analytic semigroup on $E$. Suppose that the function $f:\dom (A^{\alpha})\to E$ is Lipschitz continuous for some given $\alpha\in (0,1)$. Then for each $x\in \dom (A^{\alpha})$ the Cauchy problem
	\[ u'(t)+Au(t)=f(u(t)) \quad (t>0), \quad u(0)=x \]
	has a unique solution $u\in \Ce^1(\R_{> 0},E)\cap \Ce(\R_{\geq 0},\dom (A^\alpha))$ (see \cite[Theorem 3.3]{Pazy}). With $X=\dom (A^\alpha)$, and scrutinizing Section 3 in \cite{Pazy}, we immediately obtain that $\varphi(t,x)=u(t)$ (with the unique solution subject to $u(0)=x$) defines a dynamical system. 
\end{example}
\begin{example}
	Informally speaking, similarly to Example \ref{examp:analytic}, if a (partial) differential equation that admits unique, global solutions to all initial values with continuous dependence on initial data \emph{(Hadamard well-posedness)}, then this yields a dynamical system $\varphi$ via setting $\varphi(t,x)=u(t)$ (with the unique solution subject to $u(0)=x$). We refer to the Epilogue by G.~Nickel to \cite{EN00} for some thoughts about the underlying philosophy. Such examples are provided for instance by the two-dimensional Navier--Stokes equation and some reaction-diffusion equations (cf.{} Example \ref{examp:analytic} concerning the latter), see, e.g., \cite[Sections III.1, III.2]{Tema1998} or \cite[Chapters 8 and 9]{Robi2001}. 
\end{example}

We now define Koopman linearizations for such systems. A suitable \enquote{space of observables} is the vector space $\Cb(X)$ of all bounded complex-valued continuous functions on $X$. 
\begin{definition}
	For a semiflow $\varphi = (\varphi_t)_{t \geq 0}$ on $X$ we define the corresponding \emph{Koopman semigroup} $\euT_\varphi = (T_\varphi(t))_{t \geq 0}$ on $\Cb(X)$ by $T_\varphi(t) f \defeq f \circ \varphi_t$ for $f \in \Cb(X)$ and $t \geq 0$. 
\end{definition}

It is evident that $T_\varphi(t)$ is a linear operator on $\Cb(X)$,  $T_\varphi(s+t) = T_\varphi(s) T_\varphi(t)$ holds for all $s, t \geq 0$, and $T_\varphi(0) = \Id_{\Cb(X)}$. Thus, $\euT_\varphi$ is in fact a semigroup of linear operators. We now study its continuity properties and therefore have to consider a suitable topology on $\Cb(X)$.

Clearly, $\Cb(X)$ is a Banach space with respect to the supremum norm $\|\cdot\|_\infty$. However, besides the norm topology $\tau_{\|\cdot\|_\infty}$ there is a second natural topology on this space, namely the compact-open topology $\tau_\mathrm{c}$ defined by the seminorms $p_K (f) \defeq \sup_{x \in K} |f(x)|$ for $f \in \Cb(X)$ and every compact subset $K \subseteq X$. Both topologies have their disadvantages: On one hand, the norm topology is too strong to reflect the properties of $X$. On the other hand, equipping $\Cb(X)$ with the compact-open topology generally does not yield a complete topological vector space. In fact, $\Cb(X)$ is dense in the space $\Ce(X)$ of \emph{all} continuous complex-valued functions on $X$ with respect to the compact-open topology. In order to circumvent these problems we \enquote{mix} the two topologies. We refer to \cite{Jarc1981} and \cite{Scha1999} for an introduction to locally convex topologies.

\begin{definition}
	The \emph{mixed topology} $\tau_{\mathrm{m}}$ on $\Cb(X)$ is the strongest locally convex topology on $\Cb(X)$ agreeing on norm-bounded sets with the compact-open topology $\tau_\mathrm{c}$. 
\end{definition}
Such a locally convex topology indeed exists, see \cite{Wiweger,Wiweger2}, where this construction first appeared in an abstract context, or \cite[Chapters 1 and 2]{Coop1978} and \cite[Section 2.10.D]{Jarc1981}. Sometimes it is also called the \emph{strict topology} and defined explicitly via the seminorms $p_g$ on $\Cb(X)$ given by $p_g(f) \coloneqq \|fg\|_\infty$ for $f \in \Cb(X)$ where $g\colon X \rightarrow \C$ is a bounded function vanishing at infinity (i.e., for every $\varepsilon >0$ there is a compact subset $K \subseteq X$ with $|f(x)| \leq \varepsilon$ for every $x \in X \setminus K$). Equivalently, one can only consider the seminorms $p_g$ induced by positive upper semicontinuous functions $g$ (as done in \cite[Paragraph 3]{Summ1971} and \cite[Section II.1]{Coop1978}), see \cite[Theorem 2.4]{Sentilles}.

We refer to the fundamental work \cite{Sentilles} of Sentilles on strict topologies on $f \in \Cb(X)$, and note that his construction used the \v Cech--Stone compactification of $X$.

We now list some important properties of this topology (see \cite[Proposition II.1.2 and Corollary 1.9]{Coop1978} and \cite[Theorems 2.10.4 and 3.6.9]{Jarc1981}). Recall here, that $X$ is \emph{compactly generated} (or a \emph{$k$-space}) if a subset $A \subseteq X$ is closed whenever $A \cap K$ is closed for every compact subset $K \subseteq X$ (see \cite[Section XI.9]{Dugu1978}). Every locally compact space and every first-countable space is compactly generated. In particular, every metric space---and thus every normed vector space---is compactly generated. A function $f \colon X \rightarrow \C$ on a compactly generated space $X$ is continuous if and only if its restrictions to compact sets are continuous.
\begin{proposition}\label{basicprops} 
	\begin{enumerate}
		[(i)] 
		\item The following inclusions hold: $\tau_{\mathrm{c}} \subseteq \tau_{\mathrm{m}} \subseteq \tau_{\|\cdot\|_\infty}$.
		\item A set is bounded with respect to the mixed topology $\tau_{\mathrm{m}}$ if and only if it is norm bounded. 
		\item A sequence $(f_n)_{n \in \N}$ in $\Cb(X)$ converges to $f \in \Cb(X)$ with respect to $\tau_{\mathrm{m}}$ if and only if it is norm bounded and converges with respect to $\tau_{\mathrm{c}}$.

		\item The space $\Cb(X)$ is complete with respect to $\tau_{\mathrm{m}}$ if and only if $X$ is a compactly generated space. 
	\end{enumerate}
\end{proposition}
\begin{remark}
	We point out that many desirable properties of locally convex spaces cannot be expected to hold for the mixed topology. For example, $\Cb(X)$ equipped with $\tau_m$ is only barrelled or bornological if $X$ is compact (in which case $\tau_{\mathrm{c}} = \tau_{\mathrm{m}} = \tau_{\|\cdot\|_\infty}$), see \cite[Proposition II.1.2 5)]{Coop1978}  or \cite[Theorem 4.8]{Sentilles}. 
\end{remark}
From now on, $\Cb(X)$ is equipped with the mixed topology unless explicitly stated otherwise. We then obtain the following characterization of the dual space (see \cite[Section II.3]{Coop1978} or \cite[Theorem 7.6.3]{Jarc1981}). Here, we write $\Me(X)$ for the space of all regular complex Borel measures $\mu$ on $X$, where a complex Borel measure $\mu$ on $X$ is \emph{regular} if its variation $|\mu|$ satisfies 
\begin{align*}
	|\mu|(A) = \sup\{|\mu|(K) \mid K \subseteq A \textrm{ compact}\} = \inf\{|\mu|(O) \mid O \supseteq A \textrm{ open}\} 
\end{align*}
for every Borel measurable set $A \subseteq X$.

\begin{theorem}
	The map 
	\begin{align*}
		\Me(X) \rightarrow \Cb(X)', \quad \mu \mapsto \left[f \mapsto \int f \dd\mu\right] 
	\end{align*}
	is an isomorphism of vector spaces. 
\end{theorem}
Using this isomorphism we identify $\Cb(X)'$ with $\Me(X)$ in the following.

With the \enquote{right} topology at our disposal, we now investigate the continuity properties of Koopman semigroups. Recall that a mapping $m:I \rightarrow \LLL(E,F)$ from an interval $I \subseteq \R$ to the space of continuous linear operators $\LLL(E,F)$ between topological vector spaces $E$ and $F$ is 
\begin{itemize}
	\item \emph{strongly continuous} if the map $m_x:I \rightarrow E, \, t \mapsto m(t)x$ is continuous for every $x \in E$. 
	\item \emph{locally equicontinuous} if the set $m(K) \subseteq \LLL(E,F)$ is equicontinuous for every compact subset $K \subseteq I$. 
\end{itemize}
We abbreviate $\LLL(E):=\LLL(E,E)$. Note that many results for the theory of strongly continuous one-parameter semigroups on Banach spaces can be generalized to locally equicontinuous and strongly continuous semigroups on sequentially complete locally convex spaces (see \cite{Komu1968}, \cite{Ouch1973}, \cite{Demb1974}, \cite{AlKu2002} and \cite{Kraa2016}).
\begin{proposition}\label{prop:charcont} 
	Consider a semiflow $\varphi$ on $X$ and the induced Koopman semigroup $\euT_\varphi$ on $\Cb(X)$. 
	\begin{enumerate}
		[(i)] 
		\item If $\varphi$ is jointly continuous, then $\euT_\varphi$ is strongly continuous and locally equicontinuous. 
		\item If $X$ is compactly generated and $\euT_\varphi$ is strongly continuous, then $\varphi$ is jointly continuous. 
	\end{enumerate}
\end{proposition}
\begin{proof}
	Observe first that by \cref{basicprops} (iii) $\euT_\varphi$ is strongly continuous on $\Cb(X)$ if and only if 
	\begin{align*}
		\R_{\geq 0} \rightarrow \Ce(K), \quad t \mapsto T_{\varphi_t}f|_K 
	\end{align*}
	is continuous with respect to the norm topology on $\Ce(K)$ for every compact subset $K \subseteq X$ and every $f \in \Cb(X)$. By \cite[Lemma 4.16]{EFHN2015} this is equivalent to the continuity of the map 
	\begin{align*}
		\R_{\geq 0} \times K \rightarrow \C, \quad (t,x) \mapsto f(\varphi_t(x)) 
	\end{align*}
	for every compact subset $K \subseteq X$ and every $f \in \Cb(X)$. Since $X$ carries the initial topology with respect to the functions in $\Cb(X)$, this simply means that the restricted mappings 
	\begin{align*}
		\R_{\geq 0} \times K \rightarrow X, \quad (t,x) \mapsto \varphi_t(x) 
	\end{align*}
	are continuous for every compact subset $K \subseteq X$. This certainly holds if $\varphi$ is a jointly continuous semiflow. On the other hand, if $X$ is compactly generated, then so is $\R_{\geq 0} \times X$ (see \cite[Theorem 4.4]{Dugu1978}) and therefore the continuity of these mappings imply that $\varphi$ is jointly continuous.
	
	To finish the proof we show that $\euT_\varphi$ is locally equicontinuous whenever $\varphi$ is jointly continuous. We note that by \cite[Corollary I.1.7]{Coop1978}, a set $M$ of linear mappings on $\Cb(X)$ is equicontinuous if and only if for every norm bounded subset $F \subseteq \Cb(X)$ the set $M|_{F}$ of restrictions to $F$ is equicontinuous with respect to the compact-open topology. But---given $t_0 \geq 0$---for the set 
	\begin{align*}
		M_{t_0} \defeq \{T_\varphi(t)\mid t \in [0,t_0]\} 
	\end{align*}
	even more is true: It is equicontinuous with respect to the compact-open topology on the whole space $\Cb(X)$. In fact, if $K \subseteq X$ is a compact subset, then $L \defeq \varphi([0,t_0] \times K)$ is still compact and $p_K(T_\varphi(t)f) \leq p_L(f)$ for every $t \in [0,t_0]$ and every $f \in \Cb(X)$.
	\end{proof}

\begin{remark}
	When dealing with semigroups on Banach spaces with an additional locally convex topology there is a second natural approach besides looking at mixed topologies. Taking into account both topologies at once, one can use the concept of so-called \emph{bi-continuous semigroups} for which there is a rather rich theory (see, e.g., \cite{Kueh2001}, \cite{KuPhD} and \cite{BuFa2}). It follows from \cref{prop:charcont} that Koopman semigroups associated to jointly continuous semiflows on compactly generated completely regular spaces $X$ form a bi-continuous semigroup with respect to the norm and the compact-open topologies. It should also be noted that in some cases the class of bi-continuous semigroups on $\Cb(X)$ with respect to the compact-open topology is strictly larger than the one of locally equicontinuous, strongly continuous semigroups with respect to the mixed topology, see \cite[Example 4.1]{Fa09a}. However, if $X$ is a Polish space or a $\sigma$-compact, locally compact space, then both classes coincide, see \cite[Proposition 1.6, Theorem 3.1]{Fa09a}, \cite[Remark 2.5]{Fa04a}, a basis for this being the fundamental work \cite{Sentilles} of Sentilles on strict topologies.
\end{remark}

The most important tool in the theory of one-parameter operator semigroups is their \emph{generator}. Given a strongly continuous semigroup $\euT = (T(t))_{t \geq 0}$ of continuous operators on a locally convex space $E$, we introduce the operator $\delta \colon \dom(\delta) \rightarrow E$ by setting 
\begin{align*}
	\dom(\delta) &\defeq \left\{f \in E\mmid \lim_{t \rightarrow 0} \frac{T(t)f-f}{t} \textrm{ exists}\right\}, \textrm{ and }\\
	\delta f &\defeq \lim_{t \rightarrow 0} \frac{T(t)f-f}{t} \textrm{ for } f \in \dom(\delta). 
\end{align*}
For every $f \in\dom(\delta)$ the map $\R_{\geq 0} \rightarrow E, \, t \mapsto T(t)f$ is continuously differentiable with derivative $\delta T(t)f = T(t)\delta f$ for $t \in \R_{\geq 0}$ (see \cite[Proposition 1.2]{Komu1968}). Moreover, if $\euT$ is locally equicontinuous and $E$ is sequentially complete, then the generator $\delta$ 
\begin{itemize}
	\item is \emph{closed}, i.e., the graph of $\delta$ is closed in $E \times E$ (see \cite[Proposition 1.4]{Komu1968}), 
	\item is \emph{densely defined}, i.e., $\dom(\delta)$ is dense in $E$ (see \cite[Proposition 1.3]{Komu1968}), and 
	\item \emph{determines $\euT$}, i.e., if $\euS$ is a second strongly continuous and locally equicontinuous on $E$ with generator $\delta$, then $\euS = \euT$ (see \cite[Theorem 3.1]{Ouch1973}). 
\end{itemize}
In particular, the generator of the Koopman semigroup $\euT_\varphi$ associated to a jointly continuous semiflow $\varphi$ on a compactly generated space $X$ is closed and densely defined. From now on we denote this operator by $\delta_\varphi$. The following result shows that this is the same as the so-called \emph{Lie generator} of the semiflow. The proof is essentially the same as the one of \cite[Proposition 19]{Kueh2001} in the case of Polish spaces, but we include it for completeness (see also \cite[Proposition 2.4]{DoNe1996}). 
\begin{proposition}
	Let $\varphi = (\varphi_t)_{t \geq 0}$ be continuous semiflow on a compactly generated space $X$. For $f,g \in \Cb(X)$ the following assertions are equivalent. 
	\begin{enumerate}
		[(a)] 
		\item $f \in \dom(\delta_\varphi)$ and $\delta_\varphi f = g$. 
		\item For every $x \in X$ 
		\begin{align*}
			\lim_{t \rightarrow 0} \frac{f(\varphi_t(x)) -f(x)}{t} = g(x). 
		\end{align*} 
	\end{enumerate}
\end{proposition}
\begin{proof}
	Clearly, (a) implies (b). By \cite[Proposition 1.2]{Komu1968} (a) holds if and only if 
	\begin{align*}
		T_\varphi(t)f - f = \int_0^t T_\varphi (s)g\dd s 
	\end{align*}
	for every $t > 0$. However, this is equivalent to 
	\begin{align*}
		f(\varphi_t(x)) - f(x) = \int_0^t g(\varphi_s(x))\dd s 
	\end{align*}
	for all $t > 0$ and $x \in X$. Now take $x \in X$. By (b) the function $f_{x} \colon \R_{\geq 0} \rightarrow \C, \, s \mapsto f(\varphi_s(x))$ is right-sided differentiable with derivative $g(\varphi_s(x))$ for $s \in \R_{\geq 0}$. Since the right-sided derivative is continuous, the function $f_{x}$ is actually continuously differentiable and we can apply the fundamental theorem of calculus to obtain the desired identity for every $t> 0$.
\end{proof}
In the simple example of the translation $\varphi$ on $X= \R$ discussed above, the generator of the induced Koopman semigroup is given by $\delta_\varphi f =f'$ where $\dom(\delta_\varphi) = \Cb^{1}(\R)$ is the space of all bounded continuously differentiable functions on $\R$ with bounded derivative.
\begin{remark}\label{bicont} 
	The generator $\delta_\varphi$ of a continuous semiflow $\varphi$ on $X$ coincides with the one defined by K\"uhnemund in the setting of bi-continuous semigroups, see \cite{Kueh2001,KuPhD}. As a consequence, $\delta_\varphi$ is closed with respect to the norm topology on $\Cb(X)$ and for $\nu>0$ the operator $\nu\Id-\delta_\varphi$ is invertible with 
	\begin{align*}\label{eq:bicontlaplace} 
		(\nu\Id-\delta_\varphi)^{-1}f=\int_0^{\infty}\ee^{-\nu t}T_\varphi(t)f\dd t\quad\text{for every $f\in \Cb(X)$} 
	\end{align*}
	where the integral is understood with respect to the compact-open topology. Moreover, $\delta_\varphi$ is a Hille--Yosida operator. See \cite{Kueh2001}, \cite{KuPhD}.
	
	A particularly useful construction for bi-continuous semigroups $\euT = (T(t))_{t \geq 0}$ on a Banach space $E$ with generator $\delta$ is the following: Define $E_0:=\overline{\dom(\delta)}$, then $E_0$ is invariant under the semigroup operators $T(t)$, whose restrictions $T_0(t):=T(t)|_{E_0}$ define a one-parameter semigroup on $E_0$ which is strongly continuous with respect to the norm $\|\cdot\|_E$. The generator of $(T_0(t))_{t\geq 0}$ is the part $\delta_0:=\delta|_{E_0}$ of $\delta$ in $E_0$. Moreover, $\dom(\delta^2)\subseteq E_0$, and $E_0$ consists precisely of those functions $f\in E$ for which $t\mapsto T(t)f$ is strongly continuous with respect to the norm. Using this idea, one can apply the rich theory for strongly continuous semigroups on Banach spaces as described in the book \cite{EN00} of Engel and Nagel.
	
	We refer to \cite{Kueh2001}, \cite{KuPhD} and \cite{BuFa2} for all these results. 
\end{remark}

\section{Algebraic and lattice theoretic characterizations} The space $\Cb(X)$ is not only a vector space but carries additional structure. Equipped with pointwise multiplication of functions it becomes an algebra. Likewise, the pointwise modulus turns it into a complex vector lattice. It is straightforward to check that the multiplication and the modulus map are continuous with respect to $\tau_{\mathrm{m}}$ (see \cite[Proposition II.2.1]{Coop1978} for continuity of multiplication). Koopman semigroups respect both the algebraic and the lattice theoretic structure of $\Cb(X)$. But even more is true: They can be characterized precisely as the semigroups compatible with these structures.
\begin{theorem}\label{char1} 
	Suppose that $X$ is compactly generated. For a strongly continuous semigroup $\EuScript{T}$ on $\Cb(X)$ the following assertions are equivalent. 
	\begin{enumerate}
		[(a)] 
		\item There is a continuous semiflow $\varphi$ on $X$ with $\euT = \euT_\varphi$. 
		\item $\euT$ is a semigroup of unital algebra homomorphisms. 
		\item $\euT$ is a semigroup of unital lattice homomorphisms. 
	\end{enumerate}
	If these equivalent assertions hold, then the semiflow $\varphi$ in (a) is unique. 
\end{theorem}
Recall here that an operator $T \in \LLL(\Cb(X))$ is 
\begin{itemize}
	\item an \emph{algebra homomorphism} if $T(fg) = (Tf)(Tg)$ for all $f,g \in \Cb(X)$. 
	\item a \emph{lattice homomorphism} if $|Tf| = T|f|$ for all $f \in \Cb(X)$. 
	\item \emph{unital} if $T\mathbbm{1} = \mathbbm{1}$ where $\mathbbm{1} \in \Cb(X)$ is the constant $1$ function. 
\end{itemize}

We remark that the algebraic characterization is also contained in \cite[Section 3]{DoNe1993} (for metric $X$ and up to the oversight that the operators need to be unital). The result follows directly from the following theorem combined with \cref{prop:charcont}.
\begin{theorem}\label{lemma:char} 
	For $T \in \LLL(\Cb(X))$ the following assertions are equivalent. 
	\begin{enumerate}
		[(a)] 
		\item There is a continuous map $\psi \colon X \rightarrow X$ with $T= T_\psi$. 
		\item $T$ is a unital algebra homomorphism. 
		\item $T$ is a unital lattice homomorphism. 
	\end{enumerate}
	If these equivalent assertions hold, then the map $\psi$ in (a) is unique. 
\end{theorem}
\begin{proof}
	It is clear that (a) implies (b) and (c). The converse implications follows from \cite[Proposition II.2.3]{Coop1978} but for the reader's convenience we sketch the proof which is based on the following key observation: We can recover the space $X$ from the spectrum of $\Cb(X)$, i.e., the space $\mathcal{X}$ of all continuous unital algebra homomorphisms $\Cb(X) \rightarrow \C$ on $\Cb(X)$. More precisely, we obtain from \cite[Proposition 2.2]{Coop1978} that 
	\begin{align*}
		i \colon X \rightarrow \mathcal{X}, \quad x \mapsto i_x 
	\end{align*}
	with $i_x(f) \defeq f(x)$ for every $f \in \Cb(X)$ is a homeomorphism. Since the unital complex-valued lattice homomorphisms and the unital complex-valued algebra homomorphisms agree on any unital commutative C$^*$-algebra, we can also characterize $\mathcal{X}$ as the set of all unital lattice homomorphisms $\Cb(X) \rightarrow \C$. Thus, if $T \in \LLL(\Cb(X))$ is a unital algebra or lattice homomorphism, its adjoint $T'\colon \Cb(X) \rightarrow \Cb(X)$ restricts to a weak* continuous map $\mathcal{X} \rightarrow \mathcal{X}$. Via $i$ this induces a map $\varphi \colon X \rightarrow X$ with $T= T_\varphi$. Uniqueness of $\varphi$ is clear since $\Cb(X)$ separates the points of $X$. 
\end{proof}
\begin{remark}
	In view of the arguments used in the proof of \cref{lemma:char} and the classical Gelfand--Naimark and Kakutani theorems (see, e.g., \cite[Section 1.4]{Dixm1977} and \cite[Section II.7]{Scha1974}) representing every unital commutative C$^*$-algebra  or AM-space as the space of continuous functions on its (then compact) spectrum, the question arises whether one can also give an abstract characterization of the spaces $\Cb(X)$ for completely regular $X$. In fact, there is such a result due to Cooper showing that these are (up to some technicalities) the prototypes of so-called unital commutative Saks-C$^*$-algebras, i.e., unital commutative C$^*$-algebras with a suitable additional locally convex topology, see \cite[Proposition II.2.6 and Appendix A.2]{Coop1978} for a precise category theoretic equivalence. Loosely speaking, this result shows that no essential information is lost when passing to the Koopman linearization. 
\end{remark}

In the case of semiflows on compact spaces additional algebraic and lattice theoretic characterizations of Koopman semigroups are available in terms of their generators. The lattice theoretic one is due to Arendt (see \cite{Aren1982}) and inspired by the classical Kato inequality for the Laplace operator which is related to positivity of semigroups (see \cite{Kato1973}, \cite{ReSi1975}, \cite{NaUh1981}, \cite{Aren1984} and \cite{ArTe2019}). One makes use of the observation that for fixed $w,z \in \C$ the function 
\begin{align*}
	[0,\infty) \rightarrow \C, \quad t \mapsto |w+tz| 
\end{align*}
is right-sided differentiable at $0$ with derivative $\Re \hsign(\overline{w})(z)$ (see \cite[Lemma B-II.2.4]{PosOp1986}) where 
\begin{align*}
	\hsign(w)(z) \defeq 
	\begin{cases}
		(\sign\, w) \cdot z & \textrm{ if } w\neq 0,\\
		|z| & \textrm{ if } w = 0, 
	\end{cases}
\end{align*}
and $\sign$ is the usual complex sign function, i.e., 
\begin{align*}
	\sign(z) \coloneqq 
	\begin{cases}
		\frac{z}{|z|} &\textrm{ if } z \neq 0,\\
		0 & \textrm{ if } z =0. 
	\end{cases}
\end{align*}

We now generalize the algebraic and lattice theoretic characterizations of Koopman semigroups to the case of compactly generated completely regular spaces (see \cite[Theorems B-II.2.5 and B-II.3.4]{PosOp1986} for the case of compact $X$). To do so, we first recall that given a linear operator $\delta$ on a locally convex space $E$ we can always consider the \emph{adjoint linear relation} $\delta' \subseteq E' \times E'$, i.e., $(f',g') \in \delta'$ if 
\begin{align*}
	\langle \delta f , f' \rangle = \langle f, g'\rangle \textrm{ for every } f \in \dom(\delta).
\end{align*}
If $\delta$ is densely defined, this can be identified with a linear map defined on a subspace $\dom(\delta')$ of $E'$. Moreover, if $\delta$ is the generator of a strongly continuous semigroup $\euT = (T(t))_{t \geq 0}$ on a sequentially complete locally convex space $E$, then $\delta'$ is precisely the generator of the semigroup $\euT' = (T(t)')_{t \geq 0}$ of adjoint operators which is continuous with respect to the weak$^*$ topology on $E'$, see \cite[Prop. 2.1]{Komu1968}.
\begin{theorem}\label{char2} 
	Assume that $X$ is compactly generated. For the generator $\delta$ of a strongly continuous and locally equicontinuous semigroup on $\Cb(X)$ the following assertions are equivalent. 
	\begin{enumerate}
		[(a)] 
		\item There is a continuous semiflow $\varphi$ on $X$ with $\delta = \delta_\varphi$. 
		\item $\delta$ is unital and a derivation. 
		\item $\delta$ is unital and satisfies Kato's equality. 
	\end{enumerate}
	If these equivalent assertions hold, then the semiflow $\varphi$ in (a) is unique. 
\end{theorem}
Here, we say that a densely defined operator $\delta$ on $\Cb(X)$ 
\begin{itemize}
	\item is a \emph{derivation} if $\dom(\delta)$ is a subalgebra of $\Cb(X)$ and 
	\begin{align*}
		\delta(f g) = \delta(f)g + f\delta(g) 
	\end{align*}
	for all $f,g \in \dom(\delta)$. 
	\item satisfies \emph{Kato's equality} if 
	\begin{align*}
		\langle \Re\, (\hsign(\overline{f}) (\delta f)), \mu \rangle = \langle |f|, \delta'\mu\rangle 
	\end{align*}
	for all$f \in \dom(\delta)$ and $\mu \in \dom(\delta')$. 
	\item is \emph{unital} if $\mathbbm{1} \in \ker(\delta)$.
	\end{itemize}

We remark once again that for complete metric spaces $X$ the algebraic characterization can be found in \cite{DoNe1993} and \cite{Kueh2001} with the unitality of the operator missing.

For the proof of \cref{char2} we need several preliminary results. The following technical lemma is a generalization of \cite[Lemma B.16]{EN00}. 

\begin{lemma}\label{prodrule} 
	Let $I \subseteq \R$ be an interval with non-empty interior, $E$ and $F$ topological vector spaces and $t_0 \in I$. Let further $H \colon I \rightarrow \LLL(E,F)$ be a strongly continuous and locally equicontinuous map and $h \colon I \rightarrow E$ continuous. If 
	\begin{itemize}
		\item $h$ is differentiable in $t_0$, and 
		\item $H \cdot h(t_0) \colon I \rightarrow E ,\, t \mapsto H(t)h(t_0)$ is differentiable in $t_0$,
	\end{itemize}
	then $H \cdot h \colon I \rightarrow E,\, t \mapsto H(t)h(t)$ is continuous on $I$ and differentiable in $t_0$ with derivative 
	\begin{align*}
		(H \cdot h)'(t_0) = H(t_0)h'(t_0) + (H \cdot h(t_0))'(t_0). 
	\end{align*}
	
	\end{lemma}
\begin{proof}
	We may assume that $0=t_0 \in I$ and $I$ is compact. Since the set $h(I)$ is compact, it is an easy consequence of \cite[Theorem III.4.5]{Scha1999} that $H \cdot h$ is continuous. We write 
	\begin{align*}
		t^{-1}((H \cdot h)(t) - (H\cdot h)(0)) = H(t)t^{-1}(h(t) - h(0)) + t^{-1}(H(t) - H(0))h(0) 
	\end{align*}
	for $t \in I\setminus \{0\}$. Since $H$ is locally equicontinuous and the set $\{t^{-1}(h(t) - h(0))\mid t \in I\}$ is precompact, we conclude again from \cite[Theorem III.4.5]{Scha1999} that 
	\begin{align*}
		\lim_{t \rightarrow 0} H(t)t^{-1}(h(t) - h(0)) = H(0) h'(0). 
	\end{align*}
	On the other hand, 
	\begin{align*}
		\lim_{t \rightarrow 0} t^{-1}(H(t) - H(0))h(0) = (H \cdot h(0))'(0).
	\end{align*}
\end{proof}
\begin{lemma}\label{lemdiff1} 
	Let $\delta$ be the generator of a strongly continuous and locally equicontinuous semigroup $\euT = (T(t))_{t \geq 0}$ on $\Cb(X)$. 
	\begin{enumerate}
		[(i)] 
		\item For all $f,g \in \dom(\delta)$ the map 
		\begin{align*}
			 \R_{\geq 0} \rightarrow \Cb(X), \quad t \mapsto (T(t)f) \cdot (T(t)g) 
		\end{align*}
		is differentiable with derivative $(T(t)f) \cdot (\delta T(t)g) + (\delta T(t) f) \cdot (T(t)g)$ for $t \in \R_{\geq 0}$. 
		
		\item For all $f \in \dom(\delta)$ and $\mu \in \Me(X)$ the map 
		\begin{align*}
			 \R_{\geq 0} \rightarrow \C, \quad t \mapsto \langle |T(t)f|,\mu \rangle 
		\end{align*}
		is right-sided differentiable with derivative $\langle \Re\,(\hsign(\overline{T(t)f})(\delta T(t) f)), \mu \rangle$ for $t \in \R_{\geq 0}$. 
	\end{enumerate}
\end{lemma}
\begin{proof}
	For (i) take $f,g \in \dom(\delta)$ and apply \cref{prodrule} to the maps 
	\begin{align*}
		h&\colon [0,\infty) \rightarrow \Ce_{\mathrm{b}}(X), \quad t \mapsto T(t)g,\\
		H&\colon [0,\infty) \rightarrow \LLL(\Ce_{\mathrm{b}}(X)), \quad t \mapsto M_{T(t)f} 
	\end{align*}
	with $M_{T(t)f}k \defeq (T(t)f) \cdot k$ for $k \in \Cb(X)$ and $t \in \R_{\geq 0}$. 
	
	Now consider (ii) and let $f \in \dom(\delta)$. For $x \in X$ we obtain that the map 
	\begin{align*}
		h \colon \R_{\geq 0} \rightarrow \C, \quad t \mapsto T(t)f(x) 
	\end{align*}
	is differentiable with derivative $(\delta T(t)f)(x)$ for $t \in \R_{\geq 0}$. Combining the right-dif\-fe\-ren\-ti\-a\-bi\-li\-ty of the modulus function discussed above with a chain rule (see \cite[Proposition B-II.2.3]{PosOp1986}) we conclude that 
	\begin{align*}
		[0,1] \rightarrow \C, \quad t \mapsto |T(t)f(x)| 
	\end{align*}
	is also right-sided differentiable with derivative $\Re\, (\hsign(\overline{T(t)f(x)}) (\delta T(t)f(x)))$ for $t \in \R_{\geq 0}$. Since 
	\begin{align*}
		\sup_{s \in (0,1]}\left\|\frac{|T(s+t)f|- |T(t)f|}{t}\right\|_\infty \leq \sup_{s \in (0,1]} \left\|\frac{T(s)T(t)f- T(t)f}{t}\right\|_\infty < \infty 
	\end{align*}
	for every $t \in \R_{\geq 0}$ (see \cref{basicprops} (iii)), we can apply Lebesgue's theorem to obtain (ii). 
\end{proof}
We now characterize semigroups of lattice homomorphisms on $\Cb(X)$ using similar arguments as in \cite[Theorem B-II.2.5]{PosOp1986}.
\begin{proposition}\label{kato} 
	For the generator $\delta$ of a strongly continuous and locally equicontinuous semigroup $\euT$ on $\Cb(X)$ the following assertions are equivalent. 
	\begin{enumerate}
		[(a)] 
		\item $\euT$ is a semigroup of lattice homomorphisms. 
		\item $\delta$ satisfies Kato's equality. 
	\end{enumerate}
\end{proposition}
\begin{proof}
	If (a) holds and $f \in \dom(\delta)$ and $\mu \in \Me(X)$, we obtain that 
	\begin{align*}
		[0,1] \rightarrow \C, \quad t \mapsto \langle T(t)|f|, \mu \rangle =\langle |T(t)f|, \mu \rangle
	\end{align*}
	is right-sided differentiable in $0$ with derivative $\langle \Re\,(\hsign(\overline{f})(\delta f)), \mu \rangle$ by \cref{lemdiff1} (ii). However, if $\mu\in \dom(\delta)$, we obtain that 
	\begin{align*}
		\lim_{t \rightarrow 0} t^{-1} \langle |f|, T(t)'\mu - \mu\rangle = \langle |f|,\delta'\mu\rangle 
	\end{align*}
	and consequently $\langle |f|,\delta'\mu\rangle = \langle \Re\,(\hsign(\overline{f})(\delta f)), \mu \rangle$.
	
Now assume that (b) holds and pick $f \in \dom(\delta)$, $\mu \in \dom(\delta')$. We show that the map
	\begin{align*}
		h_s\colon [0,s+1) \rightarrow \C, \quad t \mapsto \langle |T(t)f|,T(s-t)'\mu\rangle
	\end{align*}
	is right-sided differentiable with derivative zero for every $s \geq 0$. This implies that $h_s$ is constant and therefore
		 \begin{align*}
		\langle |T(s)f|, \mu \rangle = h_s(0) = h_s(s) = \langle T(s)|f|, \mu \rangle 
	\end{align*}
	for every $s \geq0$.
	Since $\dom(\delta')$ is $\sigma(\Me(X),\Cb(X))$-dense in $\Me(X)$ (see \cite[Proposition 5.2]{Koma1964}) and $\dom(\delta)$ is dense in $\Ce_\mathrm{b}(X)$, we then obtain that $|T(s)f| = T(s)|f|$ for all $f \in \Cb(X)$ and $s \geq 0$.
	
	Now take $s >0$ and observe that---by suitably replacing $f$ and $\mu$---it suffices to show that $h_s$ is right-sided differentiable in $0$ with derivative zero. Since the set
		\begin{align*}
			 \left\{t^{-1}(T(s-t)'\mu - T(s)'\mu)\colon t \in (0,1]\right\} \subseteq \Cb(X)'
		\end{align*}
	is $\sigma(\Cb(X)',\Cb(X))$-precompact, it is equicontinuous by \cite[Corollary 1.24]{Coop1978}. Thus we obtain that
		\begin{align*}
			\lim_{t \rightarrow 0} \langle |T(t)f|,  t^{-1} (T(s-t)'\mu - T(s)'\mu) \rangle = - \langle |f|, \delta'T(s)'\mu \rangle
		\end{align*}
	by \cite[Theorem III.4.5]{Scha1999}. But then
		\begin{align*}
			t^{-1}\Bigl( \langle |T(t)f|, T(s-t)'\mu \rangle - \langle |f|,T(s)'\mu\rangle \Bigr)&= \langle |T(t)f|,  t^{-1} (T(s-t)'\mu - T(s)'\mu) \rangle \\
			&+ \langle t^{-1}(|T(t)f| - |f|),T(s)'\mu \rangle   
		\end{align*}
	converges to zero for $t \rightarrow 0$ by (b) and \cref{lemdiff1} (ii).
\end{proof}
We finally show that unitality of the semigroup operators and the generator are equivalent. This can be derived more generally for bi-continuous semigroups $(T(t))_{t\geq 0}$ by considering the restricted semigroup $(T_0(t))_{t\geq 0}$ which is strongly continuous with respect to the norm topology (cf. \cref{bicont}). 
\begin{lemma}\label{lemunital} 
	Let $\delta$ be the generator of a bi-continuous semigroup $\euT = (T(t))_{t \geq 0}$. Then
	\[ \ker(\delta)=\bigcap_{t\geq 0}\ker(\Id-T(t))=\bigcap_{t\geq 0}\ker(\Id-T_0(t)). \]

\end{lemma}
\begin{proof}
	The inclusions
	\[ \ker(\delta)\supseteq \bigcap_{t\geq 0}\ker(\Id-T(t))\supseteq\bigcap_{t\geq 0}\ker(\Id-T_0(t)) \]
	are trivial. If $f\in \ker(\delta)$, then $f\in\dom(\delta^2)\subseteq E_0$, i.e., $f\in \ker(\delta_0)$, and one can invoke Corollary 4.3.8 from \cite{EN00} to conclude
	\[ \ker(\delta)\subseteq \bigcap_{t\geq 0}\ker(\Id-T_0(t)), \]
	finishing the proof. 
\end{proof}
\begin{proof}
	[Proof of \cref{char2}] Let $\euT = (T(t))_{t \geq 0}$ be strongly continuous and locally equicontinuous semigroup generated by $\delta$. We first note that by \cref{lemunital} the operator $\delta$ is unital if and only if $T(t)$ is unital for every $t \geq 0$. 
	
	If (a) holds (and therefore (b) of \cref{char1}), we take $f,g \in \dom(\delta)$ and then obtain by \cref{lemdiff1} (i) that the map 
	\begin{align*}
		[0, \infty) \rightarrow \Cb(X), \quad t \mapsto T(t)(fg) 
	\end{align*}
	is differentiable in $0$ with derivative $f \delta(g) + \delta(f)g$. Thus, $fg \in \dom(\delta)$ with $\delta(fg) = f \delta(g) + \delta(f)g$.
	
	Suppose conversely that (b) holds and take $f,g \in \dom(\delta)$ as well as $s > 0$. By \cref{lemdiff1} (i) we obtain that 
	\begin{align*}
		h \colon [0,s] \rightarrow \Cb(X), \quad t \mapsto (T(t)f) \cdot (T(t)g) 
	\end{align*}
	is differentiable with derivative $(T(t)f) \cdot (\delta T(t)g) + (\delta T(t) f) \cdot (T(t)g)$ for $t \in [0,\infty)$. Applying \cref{prodrule} with $H(t) \defeq T(s-t)$ for $t \in [0,s]$ and (b) yield that 
	\begin{align*}
		H \cdot h\colon [0,s] \rightarrow \Cb(X), \quad t \mapsto T(s-t)((T(t)f)\cdot (T(t)g)) 
	\end{align*}
	is differentiable with derivative 
		\begin{align*}
			(H \kern-1pt\cdot\kern-1pt h)'(t) &=T(s-t)[(T(t)f) \kern-1pt\cdot\kern-1pt (\delta T(t)g) + (\delta T(t) f) \kern-1pt\cdot\kern-1pt (T(t)g)] - T(s-t)\delta ((T(t)f) \kern-1pt\cdot\kern-1pt (T(t)g))\\
			&= T(s-t)[(T(t)f) \kern-1pt\cdot\kern-1pt (\delta T(t)g) + (\delta  T(t)f) \kern-1pt\cdot\kern-1pt (T(t)g) - \delta ((T(t)f) \kern-1pt\cdot\kern-1pt (T(t)g))]\\
			&= 0
		\end{align*}
	for $t \in [0,s]$.
	Thus, 
	\begin{align*}
		T(s)(fg) = (H \cdot h)(0) = (H \cdot h)(s) = (T(s)f)(T(s)g) 
	\end{align*}
	which shows (a) by density of $\dom(\delta)$ and \cref{char1}.
	
	The equivalence of (a) and (c) immediately follows from \cref{char1} and \cref{kato}. Moreover, uniqueness of $\varphi$ follows from \cref{char1} and the fact that the semigroup is uniquely determined by its generator.
\end{proof}

\section{Attractors and stability: An outlook} Koopman semigroups have proven to be an effective tool to study dynamical systems on compact spaces. The characterization of these semigroups and their generators on spaces $\Cb(X)$ can be the starting point to extend the Koopman theory to a more general framework, thereby allowing for an approach to non-linear systems arising from partial differential equations based on \emph{linear} functional analysis and operator theory. One aspect that can be addressed is the study of attractors, playing a key role for dynamical systems and, in particular, for the solutions of partial differential equations (see, e.g., \cite[Chapter 1]{Tema1998}, \cite[Part III]{Robi2001}, \cite{GLMY2018} and \cite{GLMY2018b}). Various notions can be found in the literature and these can be translated to stability properties of the corresponding Koopman semigroup as discussed in \cite{Kueh2020} (and \cite{Kueh2019} for dynamics on compact spaces). The basis for this approach is provided by the following observation. Recall that a linear subspace $I \subseteq \Cb(X)$ is an \emph{ideal} if $fg \in I$ whenever $f \in \Cb(X)$ and $g \in I$. 
\begin{proposition}\label{correspondence} 
	The mapping 
	\begin{align*}
		M \mapsto I_M \defeq \{f \in \Cb(X)\mid f|_M = 0\} 
	\end{align*}
	establishes an inclusion reversing bijection between the set of all closed subsets of $X$ and the set of closed ideals in $\Cb(X)$. Moreover, if $\varphi = (\varphi_t)_{t \geq 0}$ is a continuous semiflow on $X$, the following assertions are equivalent for a closed subset $M \subseteq X$. 
	\begin{enumerate}
		[(a)] 
		\item $M$ is \emph{$\varphi$-invariant}, i.e., $\varphi_t(M) \subseteq M$ for every $t \geq 0$. 
		\item $I_M$ is \emph{$\euT_\varphi$-invariant}, i.e., $T_\varphi(t)I_M \subseteq I_M$ for every $t \geq 0$. 
	\end{enumerate}

\end{proposition}
\begin{proof}
	The first part readily follows from \cite[Proposition 2.7]{Coop1978}. The second part is a simple application of complete regularity of $X$ analogous to the proof of \cite[Lemma 4.18]{EFHN2015}. We omit the details. 
\end{proof}
With this correspondence, many concepts of attractiveness for a set $M \subseteq X$ can be described by stability conditions for the semigroup $\euT_\varphi|_{I_M}$ of the restricted operators $T_\varphi(t)|_{I_M}$ for $t \geq 0$. Let $\calB\subseteq \pow(X)$ be a family of subsets of $X$. A compact invariant set $M \subseteq X$ is called \emph{$\calB$-attractive} if for every $B\in \calB$ and every open neighbourhood $U$ of $M$ there is $t_0\geq 0$ with $\varphi_t(B)\subseteq U$ for every $t \geq t_0$. If $(X,d)$ is a metric space, this is equivalent to the requirement $\lim_{t \rightarrow \infty} d(\varphi_t(B),M) = 0$ for each $B\in\calB$. It can be easily proven that a compact invariant set $M \subseteq X$ is $\calB$-attractive if and only if $\lim_{t \rightarrow \infty} T_\varphi(t)f = 0$ uniformly on each $B\in \calB$ for every $f \in I_M$. A simple example in the case of $\mathscr{B} = \{X\}$ is the following.

\begin{example}
	Let $X = [0,\infty]$ be the one-point compactification of $\R_{\geq 0}$ and $\varphi_t(x) = x+t$ for $x \in X$ and $t \geq 0$ (where $\infty + t = \infty)$. Then every subset $[x,\infty]$ for $x \in X$ is $\{X\}$-attractive. 
\end{example}

In particular, we observe that in this example there exists a smallest, $\calB$-attractive, compact, invariant set (namely $\{\infty\}$). The existence of such a \enquote{smallest attractor} is an interesting problem. For example, if $\calB$ is the collection of all bounded subsets of a Banach space $X$, then the existence of a smallest, $\calB$-attractive, compact, invariant set can be shown in many examples, e.g., for the semiflows arising from the two-dimensional Navier--Stokes equation (see, e.g., \cite[Theorems 12.3 and 12.5]{Robi2001}) or certain reaction-diffusion equations (see, e.g., \cite[Theorem 11.4]{Robi2001}). A sufficient criterion for existence which is often used in this context is the following: If there is a compact \emph{absorbing set} $A\subseteq X$, i.e., for every bounded set $B \subseteq X$ there is $t_0\geq 0$ with $\varphi_t(B) \subseteq A$ for every $t \geq t_0$, then there exists a smallest, boundedly-attractive, compact, invariant set. We give a short proof of a generalization of this existence theorem for a large class of collections $\calB\subseteq \pow(X)$ and completely regular spaces $X$ by using the Koopman linearization. 
\begin{theorem}
	Let $(X;\varphi)$ be a topological dynamical system and let $\calB\subseteq\pow(X)$ such that 
	\begin{itemize}
		\item $\bigcup\calB\neq\emptyset$, and 
		\item for every $B \in \calB$ and $t \geq 0$ there is $C \in \calB$ with $\varphi_t(B) \subseteq C$. 
	\end{itemize}
	Suppose that the non-empty compact set $A$ \emph{is $\calB$-absorbing}, i.e., has the property that for every $B\in \calB$ there is $t_0 \geq 0$ with $\varphi_t(B) \subseteq A$ for every $t \geq t_0$. Then there exists a smallest, non-empty, compact, $\calB$-attractive, invariant set $M$ and 
	\begin{align*}
		I_M = \{f \in \Cb(X)\mid \lim_{t \rightarrow \infty} T_\varphi(t)f = 0\quad\text{uniformly on each $B\in\calB$}\}. 
	\end{align*}
\end{theorem}
\begin{proof}
	It is readily checked that 
	\begin{align*}
		I \defeq \{f \in \Cb(X)\mid \lim_{t \rightarrow \infty} T_\varphi(t)f = 0\quad\text{uniformly on each $B\in\calB$}\} 
	\end{align*}
	is $\euT_\varphi$-invariant ideal. By $\bigcup\calB\neq\emptyset$, we have $\mathbbm{1}\not\in I$. We prove that $I$ is closed. Take $f \in \overline{I}$, $B \in\calB$ and $\varepsilon > 0$. We then find $t_0> 0$ with $\varphi_t(B) \subseteq A$ for every $t \geq t_0$ and $g \in I$ with $p_A(f-g) \leq \varepsilon$. Then 
	\begin{align*}
		\sup_{x \in B} |T_\varphi(t)f(x) - T_\varphi(t)g(x)|\leq \varepsilon 
	\end{align*}
	for every $t \geq t_0$. Since we find such a $g \in I$ for every $\varepsilon>0$, we obtain $f \in I$. By \cref{correspondence} we now find a closed invariant set $M \subseteq X$ with $I =I_M$; by the above, $M\neq\emptyset$. If $\tilde{M}$ is any $\calB$-attractive set, then $I_{\tilde{M}} \subseteq I_M$ and therefore $M \subseteq \tilde{M}$. In particular, $M \subseteq A$ which shows that $M$ is compact. 
\end{proof}
We remark that---by considering these general collections $\calB$---our result also covers many kinds attractors (e.g., uniform attraction of the whole space, compact sets, all points or only the orbit of a single point). Moreover, the generality of completely regular spaces also allows to consider weaker forms (e.g., by considering Banach spaces equipped with the weak topology). Finally, we remark that the result can be proved analogously, with the obvious adjustments, in the time-discrete case, i.e., for Koopman operators induced by a single continuous map $\psi \colon X \rightarrow X$.

A systematic approach to attractors on completely regular spaces using the Koopman linearization therefore seems fruitful. Other promising topics for the Koopman theoretic approach are Lyapunov functions and spectral theory, which the authors aim to investigate in future works. 


\begin{thebibliography}{10}
\expandafter\ifx\csname urlstyle\endcsname\relax
  \providecommand{\doi}[1]{(doi:\discretionary{}{}{}#1)}\else
  \providecommand{\doi}{(doi:\discretionary{}{}{}\begingroup
  \urlstyle{rm}\Url)}\fi

\bibitem{Koop1931}
Koopman BO. 1931 Hamiltonian systems and transformation in {H}ilbert space.
\newblock \emph{Proc. Natl. Acad. Sci. U.S.A.} \textbf{18}, 315--318.

\bibitem{EFHN2015}
Eisner T, Farkas B, Haase M, Nagel R. 2015 \emph{Operator Theoretic Aspects of
  Ergodic Theory}.
\newblock Springer.

\bibitem{PosOp1986}
Arendt W, Grabosch A, Greiner G, Moustakas U, Nagel R, Schlotterbeck U, Groh U,
  Lotz HP, Neubrander F. 1986 \emph{One-parameter Semigroups of Positive
  Operators}.
\newblock Springer.

\bibitem{Dixm1977}
Dixmier J. 1977 \emph{C$^*$-Algebras}.
\newblock North Holland.

\bibitem{Scha1974}
Schaefer HH. 1974 \emph{Banach Lattices and Positive Operators}.
\newblock Springer.

\bibitem{DoNe1993}
Dorroh JR, Neuberger JW. 1993 Lie generators for semigroups of transformations
  on a {P}olish space.
\newblock \emph{Electron. J. Differential Equations} \textbf{1993}, 1--7.

\bibitem{DoNe1996}
Dorroh JR, Neuberger JW. 1996 A theory of strongly continuous semigroups in
  terms of {L}ie generators.
\newblock \emph{J. Funct. Anal.} \textbf{136}, 114--126.
\newblock \doi{10.1006/jfan.1996.0023}

\bibitem{DoNe2000}
Dorroh JR, Neuberger JW. 2000 Linear extensions of nonlinear semigroups.
\newblock In \emph{Semigroups of Operators: Theory and Applications} (ed.
  AV~Balakrishnan), pp. 96--102. Birkh{\"a}user Basel.

\bibitem{Kueh2001}
Kühnemund F. 2001 A {H}ille-{Y}osida theorem for bi-continuous semigroups.
\newblock \emph{Semigroup Forum} \textbf{67}, 205--225.

\bibitem{Kueh2020}
Kühner V. 2020 \emph{Koopmanism for Attractors in Dynamical Systems}.
\newblock Ph.D. thesis, Tübingen University.

\bibitem{Miya1992}
Miyadera I. 1992 \emph{Nonlinear Semigroups}.
\newblock American Mathematical Society.

\bibitem{Tema1998}
Temam R. 1998 \emph{Infinite-Dimensional Systems in Mechanics and Physics}.
\newblock Springer, 2nd edition edition.

\bibitem{Robi2001}
Robinson JC. 2001 \emph{Infinite-Dimensional Systems. An Introduction to
  Dissipative Parabolic PDEs and the Theory of Global Attractors}.
\newblock Cambridge University Press.

\bibitem{GLMY2018}
Guo B, Ling L, Ma Y, Yang H. 2018 \emph{Infinite-Dimensional Dynamical Systems.
  Volume 1: Attractors and Inertial Manifolds}.
\newblock de Gruyter.

\bibitem{GoHe1955}
Gottschalk WH, Hedlund GA. 1955 \emph{Topological Dynamics}.
\newblock American Mathematical Society.

\bibitem{Ausl1988}
Auslander J. 1988 \emph{Minimal Flows and their Extensions}.
\newblock Elsevier.

\bibitem{Bron1979}
Bron\v{s}te\v{\i}n IU. 1979 \emph{Extensions of Minimal Transformation Groups}.
\newblock Sijthoff \& Noordhoff.

\bibitem{deVr1993}
de~Vries J. 1993 \emph{Elements of Topological Dynamics}.
\newblock Springer.

\bibitem{CranLig}
Crandall MG, Liggett TM. 1971 Generation of semi-groups of nonlinear
  transformations on general {B}anach spaces.
\newblock \emph{Amer. J. Math.} \textbf{93}, 265--298.
\newblock \doi{10.2307/2373376}

\bibitem{BCP}
Benilan P, Crandall MG, Pazy A \emph{Nonlinear Evolution Equations in {B}anach
  Spaces}.
\newblock Unpublished book manuscript

\bibitem{Barbu}
Barbu V. 1976 \emph{Nonlinear Semigroups and Differential Equations in {B}anach
  Spaces}.
\newblock Editura Academiei Republicii Socialiste Rom\^{a}nia, Bucharest;
  Noordhoff International Publishing, Leiden.
\newblock Translated from the Romanian

\bibitem{Pazy}
Pazy A. 1983 \emph{Semigroups of Linear Operators and Applications to Partial
  Differential Equations}, volume~44 of \emph{Applied Mathematical Sciences}.
\newblock Springer-Verlag, New York.
\newblock \doi{10.1007/978-1-4612-5561-1}

\bibitem{EN00}
Engel KJ, Nagel R. 2000 \emph{One-Parameter Semigroups for Linear Evolution
  Equations}.
\newblock Springer.

\bibitem{Jarc1981}
Jarchow H. 1981 \emph{Locally Convex Spaces}.
\newblock Teubner.

\bibitem{Scha1999}
Schaefer HH. 1999 \emph{Topological Vector Spaces}.
\newblock Springer, 2nd edition edition.

\bibitem{Wiweger}
Wiweger A. 1961 Linear spaces with mixed topology.
\newblock \emph{Studia Math.} \textbf{20}, 47--68.
\newblock \doi{10.4064/sm-20-1-47-68}

\bibitem{Wiweger2}
Wiweger A. 1957 A topologisation of {S}aks spaces.
\newblock \emph{Bull. Acad. Polon. Sci. Cl. III.} \textbf{5}, 773--777, LXVII.

\bibitem{Coop1978}
Cooper JB. 1978 \emph{Saks Spaces and Applications to Functional Analysis}.
\newblock North-Holland.

\bibitem{Summ1971}
Summers WH. 1971 The general complex bounded case of the strict weighted
  approximation problem.
\newblock \emph{Math. Ann.} \textbf{192}, 90--98.

\bibitem{Sentilles}
Sentilles FD. 1972 Bounded continuous functions on a completely regular space.
\newblock \emph{Trans. Amer. Math. Soc.} \textbf{168}, 311--336.
\newblock \doi{10.2307/1996178}

\bibitem{Dugu1978}
Dugundji J. 1978 \emph{Topology}.
\newblock Allyn and Bacon.

\bibitem{Komu1968}
Komura T. 1968 Semigroups of operators in locally convex spaces.
\newblock \emph{J. Funct. Anal.} \textbf{2}, 258--296.

\bibitem{Ouch1973}
Ouchi S. 1973 Semi-groups of operators in locally convex spaces.
\newblock \emph{J. Math. Soc. Japan} \textbf{25}, 265--276.

\bibitem{Demb1974}
Dembart B. 1974 On the theory of semigroups of operators on locally convex
  spaces.
\newblock \emph{J. Funct. Anal.} \textbf{16}, 123--160.

\bibitem{AlKu2002}
Albanese A, Kühnemund F. 2002 Trotter--{K}ato approximation theorems for
  locally equicontinuous semigroups.
\newblock \emph{Riv. Math. Univ. Parma (7)} \textbf{1}, 19--53.

\bibitem{Kraa2016}
Kraaij R. 2016 Strongly continuous and locally equi-continuous semigroups on
  locally convex spaces.
\newblock \emph{Semigroup Forum} \textbf{92}, 158--185.

\bibitem{KuPhD}
K{\"u}hnemund F. 2001 \emph{Bi-continuous Semigroups on Spaces with Two
  Topologies: {T}heory and Applications}.
\newblock Ph.D. thesis, Eberhard-Karls-Universit\"at T\"ubingen.

\bibitem{BuFa2}
Budde C, Farkas B. 2019 Intermediate and extrapolated spaces for bi-continuous
  operator semigroups.
\newblock \emph{Journal of Evolution Equations} \textbf{19}, 321--359.
\newblock \doi{10.1007/s00028-018-0477-8}

\bibitem{Fa09a}
Farkas B. 2011 Adjoint bi-continuous semigroups and semigroups on the space of
  measures.
\newblock \emph{Czech. J. Math.} \textbf{61}, 309--322.

\bibitem{Fa04a}
Farkas B. 2004 Perturbations of bi-continuous semigroups with applications to
  transition semigroups on {$C\sb b(H)$}.
\newblock \emph{Semigroup Forum} \textbf{68}, 87--107.

\bibitem{Aren1982}
Arendt W. 1982 Kato's equality and spectral decomposition for positive
  {${C_0}$}-groups.
\newblock \emph{Manuscripta Math.} \textbf{40}, 277 -- 298.

\bibitem{Kato1973}
Kato T. 1973 Schrödinger operators with singular potentials.
\newblock \emph{Israel J. Math.} \textbf{13}, 135--148.

\bibitem{ReSi1975}
Reed M, Simon B. 1975 \emph{Method's of Modern Mathematical Physics. Vol. II}.
\newblock Academic Press.

\bibitem{NaUh1981}
Nagel R, Uhlig H. 1981 An abstract {K}ato inequality for generators of positive
  operators semigroups on banach lattices.
\newblock \emph{J. Operator Theory} \textbf{6}, 113--123.

\bibitem{Aren1984}
Arendt W. 1984 Kato's inequality: a characterization of generators of positive
  semigroups.
\newblock \emph{Proc.R. Irish Acad. Sect. A} \textbf{84}, 155--175.

\bibitem{ArTe2019}
Arendt W, ter Elst AFM. 2019 Kato's inequality.
\newblock In \emph{Analysis and Operator Theory. {D}edicated in {M}emory of
  {T}osio {K}ato’s 100th Birthday}, pp. 47--60. Springer.

\bibitem{Koma1964}
Komatsu H. 1964 Semi-groups of operators in locally convex spaces.
\newblock \emph{J. Math. Soc. Japan} \textbf{16}, 230--262.

\bibitem{GLMY2018b}
Guo B, Ling L, Ma Y, Yang H. 2018 \emph{Infinite-Dimensional Dynamical Systems.
  Volume 2: Attractors and Methods}.
\newblock de Gruyter.

\bibitem{Kueh2019}
Kühner V. 2019 What can {K}oopmanism do for attractors in dynamical systems?
\newblock \emph{J. Anal.} pp. 1--23.
\newblock \doi{10.1007/s41478-019-00211-2}

\end{thebibliography}

\end{document}